\documentclass[11pt]{amsart}
\usepackage{}
\usepackage{amsfonts}
\usepackage{amsfonts,latexsym,rawfonts,amsmath,amssymb,amsthm}
\usepackage[plainpages=false]{hyperref}
\usepackage{mathrsfs}
\RequirePackage{amsmath} \RequirePackage{amssymb}
\usepackage{fancyhdr}
\usepackage{graphicx}
\usepackage{amscd,latexsym,amsthm,amsfonts,amssymb,amsmath,amsxtra}
\usepackage{amscd,mathrsfs,latexsym,amsmath,amsthm,amssymb,amsxtra,}
\usepackage{latexsym,amsmath,amsthm,amssymb,amsxtra,mathrsfs}
\usepackage[all]{xy}



\numberwithin{equation}{section}

\newcommand{\beq}{\begin{equation}}
\newcommand{\eeq}{\end{equation}}
\newcommand{\beqs}{\begin{eqnarray*}}
\newcommand{\eeqs}{\end{eqnarray*}}
\newcommand{\beqn}{\begin{eqnarray}}
\newcommand{\eeqn}{\end{eqnarray}}
\newcommand{\beqa}{\begin{array}}
\newcommand{\eeqa}{\end{array}}

\def\lra{\longrightarrow}

\def\bc{\begin{center}}
\def\ec{\end{center}}

\def\begeq{\begin{equation}}
\def\endeq{\end{equation}}
\def\and{\quad{\rm and}\quad}

\let\lra=\longrightarrow

\def\mapright\#1{\,\smash{\mathop{\lra}\limits^{\#1}}\,}

\newtheorem{prop}{Proposition}[section]
\newtheorem{theo}[prop]{Theorem}
\newtheorem{lem}[prop]{Lemma}

\newtheorem{cor}[prop]{Corollary}
\newtheorem{rem}[prop]{Remark}

\newtheorem{defi}[prop]{Definition}

\begin{document}
\bibliographystyle{plain}

\date{}
\author {Yuxing $\text{Deng}^{*}$ }
\author { Xiaohua $\text{Zhu}^{**}$}

\thanks {*Partially supported by the NSFC 11701030, **by the NSFC 11331001 and 11771019}
\subjclass[2000]{Primary: 53C25; Secondary: 53C55,
58J05}
\keywords { Ricci flow, steady Ricci solitons,  rotational symmetry}

\address{ Yuxing Deng\\School of Mathematics and Statistics, Beijing Institute of Technology,
Beijing, 100081, China\\
6120180026@bit.edu.cn}

\address{ Xiaohua Zhu\\School of Mathematical Sciences and BICMR, Peking University,
Beijing, 100871, China\\
xhzhu@math.pku.edu.cn}

\title{Classification of  gradient steady Ricci solitons with linear curvature decay}
\maketitle

\section*{\ }

\begin{abstract} In this paper, we  give a description for   steady Ricci solitons with a linear decay of sectional  curvature. In particular, we classify all 3-dimensional steady Ricci solitons and 4-dimensional $\kappa$-noncollpased steady Ricci solitons with nonnegative sectional curvature under  the linear curvature decay.

\end{abstract}

\section{Introduction}

  In his celebrated paper \cite{ Pe1},   Perelman conjectured that \textit{any 3-dimensional $\kappa$-noncollapsed and non-flat steady (gradient) Ricci soliton  must be rotationally symmetric}.
The conjecture was solved by Brendle in 2012 \cite{Br1}. For  higher dimensional steady Ricci solitons with positive  sectional  curvature,  Brendle also proved that they must be  rotationally symmetric if they  are asymptotically   cylindrical \cite{Br2}.
By verifying  the asymptotically   cylindrical  property,  the authors recently show  in \cite{DZ6} that any higher dimensional  $\kappa$-noncollapsed  steady Ricci solitons  with  nonnegative  curvature  operator  must be  rotationally symmetric, if its scalar curvature $R(x)$ satisfies,
\begin{align}\label{scalar curvature linear decay}
R(x)\le\frac{C}{\rho(x)},~\forall~\rho(x)\ge r_0,
\end{align}
for some  $r_0>0$.
where  $\rho(x)$ denotes the distance from a fixed  point $x_0$.

In the present paper, we study steady Ricci solitons with nonnegative sectional curvature under the assumption (\ref{scalar curvature linear decay}) even without the $\kappa$-noncollapsed condition.  By using Ricci flow method, we
shall  prove

\begin{theo}\label{dimension reduction theorem}
Let $(M,g)$ be an $n$-dimensional steady Ricci soliton with nonnegative sectional curvature which satisfying  (\ref{scalar curvature linear decay}).
Then the universal cover $(\widetilde{M},\widetilde{g})$ of $(M,g)$ is one of the following:

(i) $(\widetilde{M},\widetilde{g})$ is the  euclidean space  $(\mathbb{R}^{n}, g_{Euclid})$.

(ii) $(\widetilde{M}, \widetilde{g})=(\mathbb{R}^2, g_{cigar})\times(\mathbb{R}^{n-2}, g_{Euclid})$, where $g_{cigar}$ is a Cigar solution of Ricci soliton on   $\mathbb{R}^2$.

(iii) $(\widetilde{M}, \widetilde{g})=(N^k, g_{N})\times(\mathbb{R}^{n-k}, g_{Euclid})$ ($k>2$), where $(N, g_{N})$ is a $k$-dimensional steady Ricci soliton with nonnegative sectional curvature  and positive Ricci curvature. Moreover, the scalar curvature $R_{N}(\cdot)$ of $(N, g_{N})$ satisfies
\begin{align}\label{exact linear decay}
\frac{C_1^{-1}}{\rho_N(x)}\le R_N(x)\le\frac{C_1}{\rho_N(x)}, ~ R_N(x)>>1.
\end{align}
\end{theo}

It is known by a result of Chen   that any  3-dimensional ancient solution has nonnegative sectional curvature \cite{Ch}. Thus  by Theorem \ref{dimension reduction theorem} together with  a result in \cite{DZ5},     we give the following classification  of 3-dimensional steady Ricci solitons.

\begin{cor}\label{cor-classificaion of 3d}
Let $(M,g)$ be a $3$-dimensional steady Ricci soliton which satisfies  (\ref{scalar curvature linear decay}).
Then the universal cover $(\widetilde{M},\widetilde{g})$ of $(M,g)$ is one of the following:

(i) $(\widetilde{M},\widetilde{g})$ is   the  euclidean space  $(\mathbb{R}^{3}, g_{Euclid})$.

(ii) $(\widetilde{M},\widetilde{g})=(\mathbb{R}^2,g_{cigar})\times\mathbb{R}$.

(iii) $(\widetilde{M},\widetilde{g})$ is rotationally symmetric and $(M,g)=(\widetilde{M},\widetilde{g})$.
\end{cor}

Corollary \ref{cor-classificaion of 3d}  gives a partial answer to a   conjecture  of  Hamilton that \textit{there should exist a family of collapsed 3-dimensional complete gradient steady Ricci solitons with positive curvature and $S^1$-symmetry} (cf. \cite{C-He}). Our result shows that
the curvature of  Hamilton's examples could not  have a linear decay.

In general, it is hard to classify steady Ricci solitons with a linear decay of sectional curvature in higher dimensions.
 For examples,  besides Bryant's  Ricci solitons as the rotationally symmetric solutions,   Cao constructed  a family of  U(n)-invariant steady Ricci solitons with positive sectional curvature and linear curvature decay \cite{Ca}.    Dancer and Wang   also  constructed many $\kappa$-noncollapsed steady Ricci solitons  in \cite{DW}  with  Ricci curvature ${\rm Ric}(\cdot)$ and  curvature tensor ${\rm Rm}(\cdot)$  which  satisfying
  \begin{align}\label{uniform curvature decay}
{\rm Ric}(x)\ge 0~{\rm and}~|{\rm Rm}(x)|\le \frac{C}{\rho(x)},~\forall~\rho(x) \ge r_0,
\end{align}
 for some $r_0>0$.
However,   by following the argument in the proof of Theorem \ref{dimension reduction theorem},  we  further prove

\begin{theo}\label{theo-small decay constant}
Let $(M,g)$ be an $n$-dimensional non-flat steady Ricci soliton with nonnegative sectional curvature and normalized scalar curvature
  $$\sup_{x\in M}R=1.$$
  Then, there exists a constant $\varepsilon(n)$ depends only on $n$ such that the universal cover of $(M,g)$ is $(\mathbb{R}^2,g_{cigar})\times(\mathbb{R}^{n-2},g_{Euclid})$ if
\begin{align}\label{decay with constant e(n)}
R(x)\rho(x)\le \varepsilon(n),~\forall~\rho(x)\ge r_0,
\end{align}
for some  $r_0>0$.
\end{theo}

Theorem \ref{theo-small decay constant} is an improvement of  Munteanu-Sung-Wang's result  \cite[Corollary 5.5]{MSW}), where they  proved  that the universal cover of non-flat steady Ricci solitons with nonnegative sectional curvature is $(\mathbb{R}^2,g_{cigar})\times(\mathbb{R}^{n-2},g_{Euclid})$ if it satisfies
\begin{align}
R(x)\rho(x)\to 0,~as~\rho(x)\to\infty.
\end{align}

 By Theorem \ref{dimension reduction theorem}, we can  also classify  non-flat  $4$-dimensional $\kappa$-noncollapsed steady Ricci solitons with  a  linear decay of  nonnegative sectional curvature.

\begin{cor}\label{theo-4 dimension noncollapsed case}
Let $(M,g)$ be a   non-flat  $4$-dimensional $\kappa$-noncollapsed steady Ricci soliton with nonnegative sectional curvature  which satisfying  (\ref{scalar curvature linear decay}).
Then $(M,g)$ must be rotationally symmetric.
\end{cor}

The proof of  Corollary \ref{theo-4 dimension noncollapsed case} is  reduced  to using  a result on  $4$-dimensional steady Ricci solitons with   nonnegative sectional curvature   in \cite{DZ6}, where   the   asymptotically   cylindrical  property is verified  in terms of Brendle  under the condition (\ref{exact linear decay}).    In fact,  for  any  $n$-dimensional  $\kappa$-noncollapsed steady Ricci solitons,   we  can weaken   the nonnegativity of sectional curvature  in  Theorem \ref{dimension reduction theorem}   to (\ref{uniform curvature decay})  to   study the  asymptotic behavior of level sets  $\Sigma_r=\{x\in M|~f(x)=r\}$  on $M$ (cf.  Theorem  \ref{theorem-uniform decay and noncollapsed case}). In particular,
when $n=4$,  we also derive the   asymptotically   cylindrical  property  as follows.

\begin{theo}\label{cor-4d-Ricci positive}
Let $(M,g)$ be a $4$-dimensional $\kappa$-noncollapsed steady Ricci soliton with positive Ricci curvature. Suppose that there are  $r_0, C>0$ such that
$$|{\rm Rm}(x)|\le \frac{C}{\rho(x)},~\forall~\rho(x) \ge r_0, $$
for some  $r_0>0$.  Then, $(M,g)$ is asymptotically cylindrical. Moreover, the sectional curvature is positive away from a compact set $K$ in $M$.
\end{theo}

Compared with Corollary \ref{theo-4 dimension noncollapsed case}, we conjecture that  the $4$-dimensional steady Ricci soliton in Theorem \ref{cor-4d-Ricci positive} is actually a Bryant soliton.

At last, we give the main idea in the proof of Theorem \ref{dimension reduction theorem}.   By a splitting result \cite[Theorem 1.1]{GLX}  (also  see the proof in  \cite[Lemma 5.1]{DZ6}),   we only need to deal with the case (iii)  in  the theorem, and to prove the split-off steady Ricci soliton $(N^k,g_{N})$ $(k\ge 3)$  with nonnegative sectional curvature  and positive Ricci curvature satisfies (\ref{exact linear decay}).   First, we show that each level set  $\Sigma_r$ $(r>>1)$ on   $(N,g_{N})$ has a uniform bounded diameter after the  scale $r^{-1}$ (cf. Section 3).  The method is to estimate the distance function along $\Sigma_r$  as done in \cite{DZ6}.  Next, we show that the rescaled   ($\Sigma_{r_i}, r_i^{-1}g_{N})$  is an almost flat manifold  for any  $i\ge i_0$ if there is a sequence of $p_i\in N$ such that $\lim_{i\to\infty}R_N(p_i)\rho_N(p_i)=0$, where $f_N(p_i)=r_i$ (cf. Section 4). Then   by the famous   Gromov  Theorem for the almost flat manifolds \cite{G} (also see Theorem \ref{Gromov's theorem}), we will derive a contradiction since  $\Sigma_r$ is diffeomorphic  to a sphere \cite[Lemma 2.1]{DZ6}.

The paper is organized as follows. In Section 2, we recall  some previous results for steady Ricci soliton with nonnegative sectional curvature in \cite{DZ6} and give a curvature estimate (cf. Proposition \ref{lem-pointwise curvature estimate}).  In Section 3, we give a diameter  estimate of level sets (cf. Proposition  \ref{theorem-diameter estimate}).  Theorem \ref{dimension reduction theorem},
Theorem \ref{theo-small decay constant} and Theorem \ref{cor-4d-Ricci positive}  are proved in Section 4,5,6, respectively.

\section{ Preliminary and  curvature estimate}\label{section 2}

 $(M,g,f)$ is called a  gradient steady Ricci soliton if  Ricci curvature  ${\rm Ric}(\cdot)$ of $g$ on $M$ satisfies
\begin{align}\label{soliton-equation}
{\rm Ric}={\rm Hess} ~f,
\end{align}
for some smooth function $f\in C^{\infty}(M)$.

Let $\phi_t$ be   one-parameter diffeomorphisms   group  generated by $-\nabla f$  on $M$.  Then $g(t)=\phi_t^{\ast}g$ satisfies Ricci flow,
\begin{align}\label{Ricci flow equation}
\frac{\partial g(t)}{\partial t}=-2{\rm Ric}(\cdot,t).
\end{align}

In this section, we always  assume the steady Ricci soliton $(M,g,f)$ has nonnegative Ricci curvature and $M$ admits  an equilibrium point $o$ such that $\nabla f(o)=0$. Note that  $R$ is nonnegative by Chen's  result  in \cite{Ch}. Then by the identity
\begin{align}\label{identity}|\nabla f|^2+R\equiv const,
\end{align}
it is easy to see that
$$R(o)=R_{\max}=\max_{x\in M} R(x).$$

The existence of  equilibrium points on $(M,g)$  is guaranteed by the following  lemma \cite[Lemma 5.1]{DZ6}.

\begin{lem}\label{lem-level set structure nonnegative case}
Let $(M,g,f)$ be a   non-flat steady Ricci soliton with nonnegative sectional curvature.  Let $S=\{p\in M|\nabla f(p)=0\}$ be equilibrium points set of  $(M,g,f)$. Suppose that
 scalar curvature $R$ of $g$ decays uniformly.    Then the following statements are true.

 \begin{enumerate}
 \item[(1)]$(S,g_S)$ is a non-empty  compact flat manifold, where $g_{S}$ is an induced metric $g$.
 \item[(2)]Let $o\in S$.  Then level set $\Sigma_r=\{x\in M|~f(x)=r\}$ is a compact hypersurface of $M$. Moreover,  each  $\Sigma_r$ is   diffeomorphic to the other  whenever  $r>f(o)$.
\item[(3)] $M_r=\{x\in M|~f(x)\le r\}$ is  compact for any $r>f(o)$.
\item[(4)] $f$ satisfies
\begin{align} \label{linear of f}
c_1\rho(x)\le f(x)\le c_2 \rho(x), ~\forall~\rho(x)\ge r_0.
\end{align}

 \end{enumerate}

\end{lem}

The constants $c_1$ and $c_2$ in (\ref{linear of f})  can be estimated  more precisely.

\begin{lem}\label{lem-precise estimate of f}
Let $(M,g,f)$ be a steady Ricci soliton with nonnegative Ricci curvature.  Suppose  that $f$ satisfies $(\ref{linear of f})$ and the scalar curvature $R$ decays uniformly. Then the following statements are true.

 \begin{enumerate}
 \item[1). ] $f(x)$ satisfies
\begin{align}\label{limit constant}
 \frac{f(x)}{\rho(x)}\to \sqrt{R_{\max}}, ~as~\rho(x)\to\infty.
\end{align}
 \item[2). ] If $R(x)\le \frac{C}{\rho(x)}$, then there are constants $C_1, C_2>0$ such that
 \begin{align}\label{two-sides}
 -C_1\sqrt{\rho(x)}+\sqrt{R_{\max}}\rho(x)\le f(x)\le \sqrt{R_{\max}}\rho(x) + C_2,~\forall~\rho\ge r_0.
 \end{align}
\end{enumerate}
\end{lem}

\begin{proof} 1).   Since $f$ satisfies $(\ref{linear of f})$, it is easy to see  by (\ref{soliton-equation}) that $R$ is not identically zero and so $R_{\max}>0$. Note that  $R$ decays uniformly.  Thus, by (\ref{identity}),
 \begin{align}\label{nonzero of gradient f-2}
 |\nabla f|^2(x)>\frac{\sqrt{R_{\max}}}{2},~\forall~f(x)\ge r_0.
 \end{align}

  Choose $r_1>0$ such that $\Sigma_{r_0}\subseteq B(o,r)$ for all $r\ge r_1$, where $o\in M$ is a fixed point.  Thus  there exists $\tau_x>0$ for any $x\in M\setminus B(o,r)$  such that $\phi_{\tau_x}(x)\in \Sigma_{r_0}$ by (\ref{nonzero of gradient f-2}).   It follows that  there is a  $t_x>0$ such that $\phi_{t_x}(x)\in\partial B(o,r)$. By integrating along the curve $\phi_s(x)$, we have
\begin{align}
f(x)-f(\phi_{t_x}(x))=\int_{0}^{t_x}|\nabla f|^2(\phi_{s}(x)) {\rm d}s,\notag
\end{align}
and
\begin{align}
d(x,\phi_{t_x}(x))\le{\rm Length}(\phi_s(x)|_{[0,t_x]},g)=\int_{0}^{t_x}|\nabla f|(\phi_{s}(x)) {\rm d}s.\notag
\end{align}
Since
\begin{align}
\frac{{\rm d}|\nabla f|^2}{{\rm d}t}(\phi_{t}(x))=-2{\rm Ric}( \nabla f,\nabla f)\le 0, ~\forall~t\in \mathbb{R}, \notag
\end{align}
\begin{align}
|\nabla f|(\phi_{s}(x))\ge |\nabla f|(\phi_{t_x}(x)), ~\forall~s\in[0,t_x].\notag
\end{align}
Consequently,
\begin{align}
f(x)-f(\phi_{t_x}(x))=&\int_{0}^{t_x}|\nabla f|^2(\phi_{s}(x)) {\rm d}s \notag\\
\ge& |\nabla f|(\phi_{t_x}(x))\int_{0}^{t_x}|\nabla f|(\phi_{s}(x)) {\rm d}s  \notag\\
\ge& |\nabla f|(\phi_{t_x}(x))d(x,\phi_{t_x}(x)).\notag
\end{align}
Note that  $\phi_{t_x}(x)\in\partial B(o,r)$.  We derive
\begin{align}
d(x,\phi_{t_x}(x))\ge \rho(x)-\rho(\phi_{t_x}(x))=\rho(x)-r\notag
\end{align}
and
\begin{align}
|\nabla f|(\phi_{t_x}(x))=\sqrt{R_{\max}-R(\phi_{t_x}(x))}\ge \sqrt{R_{\max}-R_r},\notag
\end{align}
where $R_r=\sup_{y\in \partial B(o,r)}R(y)$.   Therefore,   we obtain
\begin{align}\label{f-growth}
f(x)\ge \sqrt{R_{\max}-R_r}(\rho(x)-r)+\underline {f}(r),
\end{align}
where $\underline{f}(r)=\inf_{y\in \partial B(o,r)}f(y)$.

For $\rho(x)\gg 1$, we  take $r=\sqrt{\rho(x)}$.  Then,  by (\ref{linear of f}) and (\ref{f-growth}), we get
\begin{align}\label{precise estimate}
f(x)\ge \sqrt{R_{\max}-R_{\sqrt{\rho(x)}}}(\rho(x)-\sqrt{\rho(x)})+C'\sqrt{\rho(x)},
\end{align}
On the other hand, by integrating $f$ along the minimal geodesic $\gamma(s)$ connecting $x$ and $o$, we have
\begin{align}\label{upper bound}
f(x)-f(o)\le \sqrt{R_{\max}}\rho(x).
\end{align}
Hence, combining (\ref{precise estimate}) and (\ref{upper bound}), we obtain (\ref{limit constant}).

\item[2). ] Note that
\begin{align}
R_{\sqrt{\rho(x)}}\le \frac{C}{\sqrt{\rho(x)}}.\notag
\end{align}
Then, by (\ref{precise estimate}) and (\ref{upper bound}), we  also get (\ref{two-sides}).
\end{proof}

The following lemma is due to \cite[Lemma 2.2]{DZ6}.

\begin{lem}\label{geodesic ball in level set}
Let $o\in M$ be any fixed point  of   steady  Ricci soliton $(M,g,f)$. Then  for   any $p\in M$ and number $k>0$ with $f(p)-k\sqrt{f(p)}>f(o)$,   it holds
\begin{align}\label{set-mr-contain-1}
B(p,\frac{k}{\sqrt{R_{\max}}};  f^{-1}(p)g)\subset M_{p,k},
\end{align}
where $M_{p,k}\subset M$ is  a subset  defined  by
$$M_{p,k}=\{x\in M| ~ f(p)-k\sqrt{f(p)}\le f(x)\le f(p)+k\sqrt{f(p)}\}.$$

\end{lem}
\begin{proof}
For any $q\in M$, let $\gamma(s)$ be a minimal geodesic connecting $p$ and $q$ such that $\gamma(s_{1})=q$ and $\gamma(s_{2})=p$. Since
$$|\nabla f|^2(x)=R_{\max}-R(x)\le R_{\max},~\forall~x\in M,$$
\begin{align*}
|f(p)-f(q)|=|\int_{s_{1}}^{s_{2}}\langle\gamma^{\prime}(s),\nabla f\rangle ds|
\le\int_{s_{1}}^{s_{2}}|\nabla f|ds
\le\sqrt{R_{\max}}d(q,p)
\end{align*}
In particular,  for  $q\in M\setminus M_{p,k}$, we get
\begin{align}
d(q,p)\geq k\sqrt{f(p)}\cdot \frac{1}{\sqrt{R_{\max}}}.\notag
\end{align}
Hence
\begin{align}
B(p,\frac{k}{\sqrt{R_{\max}}}; f^{-1}(p)g)\subset M_{p,k}.
\end{align}

\end{proof}

We end this section by the following curvature estimate.

\begin{prop}\label{lem-pointwise curvature estimate}
Let $(M,g,f)$ be a steady Ricci soliton with   Ricci  curvature and sectional curvature  which satisfying (\ref{uniform curvature decay}).  Suppose that $f$ satisfies $(\ref{linear of f})$.
Then  there exists a constant $C(k)$ for each $k\in\mathbb{N}$ such that
\begin{align}\label{curvature-estimate-k}
|\nabla^k {\rm Rm}|(p)\cdot f^{\frac{k+2}{2}}(p)\le C(k),~\forall~p\in ~M.
\end{align}
\end{prop}

\begin{proof}
 Fix any $p\in M$ with $f(p)\ge 2r_0>>1$. Then
\begin{align}
|f(x)-f(p)|\le \sqrt{f(p)},~ \forall ~x\in M_{p,1}.\notag
\end{align}
By Lemma \ref{lem-precise estimate of f}, we may assume that
\begin{align}
\frac{\sqrt{R_{\max}}}{2}\rho(x)\le f(x)\le \sqrt{R_{\max}}\rho(x),~\forall~\rho(x)\ge r_0.\notag
\end{align}
Thus, by (\ref{uniform curvature decay}), it holds  for any $x\in M_{p,1}$,
\begin{align}
|{\rm Rm}(x)|\cdot f(p)\le  C\sqrt{R_{\max}}\frac{f(p)}{f(x)}\le \frac{ C\sqrt{R_{\max}}f(p)}{f(p)-\sqrt{f(p)}}\le 2C\sqrt{R_{\max}}.\notag
\end{align}
By (\ref{set-mr-contain-1}) in Lemma \ref{geodesic ball in level set}, we get
\begin{align}\label{rm-point}
|{\rm Rm}(x)|\cdot f(p)\le C^{\prime},~\forall ~x\in B(p,\frac{1}{\sqrt{R_{\max}}};g_p)\subseteq M_{p,1},
\end{align}
where $g_p=f^{-1}(p)g$.

Since
\begin{align}
\frac{\partial}{\partial t}f(\phi_t(p))=-|\nabla f|^2(x,t)\le0,\notag
\end{align}
\begin{align}
f(\phi_t(x))\ge f(x)\ge r_0,~\forall~t\le0,~x\in B(p,\frac{1}{\sqrt{R_{\max}}};g_p).\notag
\end{align}
Combining with (\ref{uniform curvature decay}), we obtain from (\ref{rm-point}),
\begin{align}
|{\rm Rm}_{g_p(t)}(x)|\le \frac{\sqrt{CR_{\max}}}{f(x)}\le\frac{C'}{f(p)},~\forall~x\in B(p,\frac{1}{\sqrt{R_{\max}}};g_p(0)),~t\in [-1,0],\notag
\end{align}
where  ${\rm Rm}_{g_p(t)}$ is the   curvature tensor of rescaled flow  $ g_p(t)=f^{-1}(p) g(f(p) t)$.

Note that   the Ricci curvature of $(M,g)$ is nonnegative when $f(x)\ge r_0$ and $g_p(t)$    satisfies   Ricci flow (\ref{Ricci flow equation}).
 Then  the  flow $g_p(t)$ is shrinking on $B(p,\frac{1}{\sqrt{R_{\max}}};g_p(0))$  by Lemma \ref{geodesic ball in level set}. Thus
\begin{align}
B(p,\frac{1}{\sqrt{R_{\max}}};g_p(t))\subseteq B(p,\frac{1}{\sqrt{R_{\max}}};g_p(0)),~\forall ~t\le0.\notag
\end{align}
By Shi's higher order estimates, we get
\begin{align}
|\nabla^k_{(g_p(t))} {\rm Rm}_{g_p(t)}|(x)\le C^{\prime}(k),~\forall~x\in B(p,\frac{1}{ 2\sqrt{R_{\max}}};g_p(-1)),~t\in  [-\frac{1}{2},0].\notag
\end{align}
It follows that
\begin{align}
|\nabla^k {\rm Rm}|(x)\cdot f^{\frac{k+2}{2}}(p)\le C^{\prime}(k),~\forall~x\in B(p,\frac{1}{2\sqrt{R_{\max}}};g_p(-1)).\notag
\end{align}
In particular, we obtain
\begin{align}
|\nabla^k {\rm Rm}|(p)\cdot f^{\frac{k+2}{2}}(p)\le C_{2}^{\prime}(k),~{\rm as}~f(p)\ge 2r_0.\notag
\end{align}
The proposition  is  proved.
\end{proof}

\begin{rem}
When $R_{\max}=1$,  according to the proof of Proposition \ref{lem-pointwise curvature estimate}, it is easy to see
by Lemma \ref{lem-precise estimate of f} that  the  constant $C(k)$ in (\ref{curvature-estimate-k}) depends only on $k$ and $C$,
where $C$ is the constant in (\ref{uniform curvature decay}).
\end{rem}

\section{diameter estimate}\label{section 3}

In this section,  we follow the argument  in \cite{DZ6}  to give a diameter estimate for steady Ricci solitons with Ricci   curvature and sectional curvature which satisfying (\ref{uniform curvature decay}).

Since $g(t)=\phi_t(g)$  is a solution of  Ricci flow (\ref{Ricci flow equation}),  the Ricci curvature $R_{ij}(\cdot, t)$  of $g(t)$  satisfies \cite{H2},
\begin{align}
\frac{\partial R_{ij}}{\partial t}=\Delta R_{ij}+2R_{kijl}R_{kl}-2R_{ik}R_{jk}.
\end{align}
Note that
\begin{align}
{L}_{X}R_{ij}=-\frac{\partial R_{ij}}{\partial t}\notag
\end{align}
and
\begin{align}
{L}_{X}R_{ij}-\nabla_X R_{ij}={\rm Ric}(\nabla_{e_i}X,e_j)+{\rm Ric}(e_i,\nabla_{e_j} X)=2R_{ik}R_{jk},\notag
\end{align}
where $X=\nabla f$. Then,
\begin{align}
\nabla_X R_{ij}+\Delta R_{ij}+2R_{kijl}R_{kl}=0.\notag
\end{align}
On the other hand, by the soliton equation (\ref{soliton-equation}), we have
\begin{align}\label{commutation}
\nabla_iR_{jk}-\nabla_jR_{ik}=\nabla_i\nabla_j\nabla_k f-\nabla_j\nabla_i\nabla_k f=-R_{ijkl}\nabla_l f.
\end{align}
Hence,
\begin{align}\label{curvature-x}
{\rm Rm}(X,e_j,e_k,X)=(\nabla {\rm Ric})(e_j,X,e_k)-(\nabla {\rm Ric})(X,e_j,e_k).
\end{align}

By taking trace of (\ref{commutation}) and the contracted Bianchi Identity, we also have
\begin{align}
{\rm Ric}(X,e_i)=-\frac{1}{2}\nabla_{e_i} R.\notag
\end{align}
Then
\begin{align}\label{nabla-ricci}
(\nabla {\rm Ric})(e_j,X,e_k)=&\frac{\partial ({\rm Ric}(X,e_k))}{\partial x_j}-{\rm Ric}(X,\nabla_{e_j}e_k)-{\rm Ric}(\nabla_{e_j}X,e_k)\notag\\
=&-\frac{1}{2}\frac{\partial^2 R}{\partial x_j\partial x_k}+\frac{1}{2}\langle\nabla_{e_j}e_k,\nabla R\rangle-R_{jl}R_{lk}\notag\\
=&-\frac{1}{2}({\rm Hess}R)_{jk}-R_{jl}R_{kl}.
\end{align}
Therefore,  combining  (\ref{curvature-x}) and (\ref{nabla-ricci}), we get

\begin{lem}\label{lem-formula of sectional curvature}
\begin{align}
{\rm Rm}(X,e_j,e_k,X)=-\frac{1}{2}({\rm Hess}R)_{jk}-R_{jl}R_{kl}+\Delta R_{jk}+2R_{ijkl}R_{il}
\end{align}
\end{lem}

\begin{lem}\label{lem-differential inequality of distance function}
Let $(M,g,f)$ be a steady Ricci soliton with   Ricci  curvature and curvature tensor  which satisfying (\ref{uniform curvature decay}).
 Suppose that $f$ satisfies $(\ref{linear of f})$.
Then for $x_1,x_2\in\Sigma$
with $d_{r}(x_1,x_2)\ge 2\tau_0$, we have
\begin{align}
\frac{\rm d}{{\rm d}r}d_r(x_1,x_2)\le C(\frac{\tau_0}{r}+\frac{1}{\tau_0}+\frac{d_{r}(x_1,x_2)}{r^{2}}),~\forall~r\ge r_0.
\end{align}
where $d_r(\cdot,\cdot)$ is the distance function of $(\Sigma,h_r)$.
\end{lem}

\begin{proof}The proof follows  from the one of  \cite[Lemma 3.1]{DZ6}. When $f(x)$ satisfies $(\ref{linear of f})$ and Ricci   curvature and  curvature tenser satisfies (\ref{uniform curvature decay}), it is easy to check that $S=\{x\in M|\nabla f(x)=0\}$ is non-empty and compact. Then, we may assume that $|\nabla f|^2(x)=R_{\max}-R(x)>0$, $\forall~f(x)>r_0$.  Thus,  every $\Sigma_r$ is diffeomorphic to a compact codimensional $1$ submanifold $\Sigma\subset M$  when $r\ge r_0$.

 As in \cite{DZ6},  we introduce  one parameter group of diffeomorphisms $F_r:\Sigma\to \Sigma_r\subseteq M$ $(r\ge r_0)$, which is generated by flow
$$\frac{\partial F_r}{\partial r}=\frac{\nabla f}{|\nabla f|^2}. $$
Let $h_r=F_r^{\ast}(g)$ and $e_i=F_\ast(\bar{e}_i), e_j=F_\ast(\bar{e}_j)$, where $\bar{e}_i,\bar{e}_j\in T\Sigma$. Then, a direct computation shows that
\begin{align}\label{diff-flow}
\frac{\partial h_r}{\partial r}(\bar{e}_i,\bar{e}_j)
=\frac{2}{|\nabla f|^2}{\rm Ric}(e_i,e_j).
\end{align}

Let $\gamma$ be a normalized minimal geodesic from $x_1$ to $x_2$ with velocity field $X(s)=\frac{{\rm d}\gamma}{{\rm d}s}$ and $V$ any piecewise smooth normal vector field along $\gamma$ which vanishes at the endpoints. By the second variation formula, we have
\begin{align}
\int_{0}^{d_r(x_1,x_2)}(|\nabla_{X}V|^2+\langle \bar R(V,X)V,X\rangle){\rm d}s\ge0.\notag
\end{align}

Let $\{e_i(s)\}_{i=1}^{n-1}$ be a parallel orthonormal frame along $\gamma$ that is perpendicular to $X$. Put $V_i(s)=f(s)e_i(s)$, where $f(s)$ is defined as
\begin{align}
f(s)=\frac{s}{\tau_0},&~\mbox{if}~ 0\le s\le \tau_0;\notag\\
f(s)=1,&~\mbox{if}~\tau_0\le s\le d_r(x_1,x_2)-\tau_0;\notag \\
f(s)=\frac{d_r(x_1,x_2)-s}{\tau_0},&~\mbox{if}~d_r(x_1,x_2)-\tau_0\le s\le d_r(x_1,x_2).\notag
\end{align}
By a direct computation, we have
\begin{align}\label{second-variation}
0\le &\sum_{i=1}^{n-1}\int_{0}^{d_r(x_1,x_2)}(|\nabla_{X}V_i|^2+\langle \bar R(V_i,X)V_i,X\rangle){\rm d}s\notag\\
=&\frac{2(n-1)}{r_0}-\int_{0}^{d_r(x_1,x_2)}\overline {\rm Ric}(X,X){\rm d}s+\int_{0}^{\tau_0}(1-\frac{s^2}{\tau_0^2})\overline {\rm Ric}(X,X){\rm d}s\notag\\
+&\int^{d_r(x_1,x_2)}_{d_r(x_1,x_2)-\tau_0}(1-\frac{(d_r(x_1,x_2)-s)^2}{\tau_0^2})\overline {\rm Ric}(X,X){\rm d}s.
\end{align}

We claim
\begin{align}\label{curvature-decay-2}
\overline {\rm Ric}(X,X)\le \frac{C}{r}\bar g(X,X),~\forall ~x\in \Sigma_r.
\end{align}

By Gauss formula, we have
\begin{align}\label{gauss-equa}
&{\rm Rm}(X,Y,Z,W)=
\overline{\rm Rm}(X,Y,Z,W)\notag\\
&+ \frac{1}{|\nabla f|^2}({\rm Ric}(X,Z){\rm Ric}(Y,W)- {\rm Ric}(X,W){\rm Ric}(Y,Z))
\end{align}
and
\begin{align}
R_{ij}=\overline{R}_{ij}+R(\frac{\nabla f}{|\nabla f|},e_{i},e_{j},\frac{\nabla f}{|\nabla f|})-\frac{1}{|\nabla f|^{2}}\sum_k(R_{ij}R_{kk}-R_{ik}R_{kj}),\notag
\end{align}
where indices $i,j,k$ are corresponding to vector fields on $T\Sigma_{r}$.
Thus for a unit vector $Y$, we derive
\begin{align}\label{ric-barric}
({\rm Ric}-\overline{\rm Ric})(Y,Y)=&R(\frac{\nabla f}{|\nabla f|},Y,Y,\frac{\nabla f}{|\nabla f|})\notag\\
&-\frac{1}{|\nabla f|^2}\sum_{i=1}^{n-1}[{\rm Ric}(Y,Y){\rm Ric}(e_i,e_i)-{\rm Ric}^2(Y,e_i)].
\end{align}

On the other hand, by Lemma \ref{lem-formula of sectional curvature}, we have
\begin{align}
&|\nabla f|^2\cdot{\rm Rm}(\frac{\nabla f}{|\nabla f|},Y,Y,\frac{\nabla f}{|\nabla f|})\notag\\
&=-\frac{1}{2}({\rm Hess }R)(Y,Y)-{\rm Ric}(Y,e_i){\rm Ric}(Y,e_i)\notag\\
&+(\Delta {\rm Ric})(e_j,e_k)+2{\rm Rm}(e_i,Y,Y,e_j){\rm Ric}(e_i,e_j).\notag
\end{align}
Note that by Lemma \ref{lem-pointwise curvature estimate},
\begin{align}
|{\rm Hess }R|+|\Delta {\rm Ric}|+|{\rm Ric}|^2+|{\rm Rm}|\cdot|{\rm Ric}|\le \frac{C_1}{f^2}.\notag
\end{align}
Then
\begin{align}
|{\rm Rm}(\frac{\nabla f}{|\nabla f|},Y,Y,\frac{\nabla f}{|\nabla f|})|\le \frac{C_1 }{|\nabla f|^2\cdot f^2}\le \frac{C}{2r^2}.\notag
\end{align}
Also, we have
\begin{align}
\frac{1}{|\nabla f|^2}|\sum_{i=1}^{n-1}[{\rm Ric}(Y,Y){\rm Ric}(e_i,e_i)-{\rm Ric}^2(Y,e_i)]|
\le \frac{R^2+|{\rm Ric}|^2}{|\nabla f|^2}\le \frac{C}{2r^2}.\notag
\end{align}
Hence, we get from (\ref{ric-barric}),
\begin{align}\label{diference-ricci}
| ({\rm Ric}-\overline{\rm Ric})(Y,Y) | \le \frac{C}{r^{2}}, ~r\ge ~r_0.
\end{align}
In particular,
\begin{align}
\overline {\rm Ric}(Y,Y)\le \frac{C_1}{ r}, ~r\ge ~r_0.\notag
\end{align}
This proves (\ref{curvature-decay-2}).

By (\ref{second-variation}) and (\ref{curvature-decay-2}), it is easy to see that
\begin{align}\label{integral-bar-ricci}
\int_{0}^{d_r(x_1,x_2)}\overline {\rm Ric}(Y,Y){\rm d}s\le\frac{2(n-1)}{\tau_0}+\frac{4C\tau_0}{3r}.
\end{align}
Also by (\ref{diference-ricci}), we see that
\begin{align}\label{ric-differen}\int_{0}^{d_r(x_1,x_2)}({\rm Ric}-\overline{\rm Ric})(Y,Y){\rm d}s\le \frac{C}{r^{2}}d_r(x_1,x_2).
\end{align}
On the other hand, if we
let $Y(s)=(F_r)_{\ast}(X(s))$ with $|Y(s)|_{(\Sigma_r,g)}\equiv 1$, then by (\ref{diff-flow}), we have
\begin{align}
&\frac{{\rm d}}{{\rm d}r}d_r(x_1,x_2)\notag\\=&\frac{1}{2}\int_{0}^{d_r(x_1,x_2)}(L_{\frac{\nabla f}{|\nabla f|^2}}g)(Y,Y){\rm d}s\notag\\
=&\frac{1}{|\nabla f|^2}\int_{0}^{d_r(x_1,x_2)}\overline{\rm Ric}(Y,Y){\rm d}s+\frac{1}{|\nabla f|^2}\int_{0}^{d_r(x_1,x_2)}({\rm Ric}-\overline{\rm Ric})(Y,Y){\rm d}s\notag.
\end{align}
Thus inserting (\ref{integral-bar-ricci}) and (\ref{ric-differen}) into the above relation,
we obtain
$$\frac{{\rm d}}{{\rm d}r}d_r(x_1,x_2)\le C(\frac{\tau_0}{r}+\frac{1}{\tau_0}+\frac{d_{r}(x_1,x_2)}{r^{2}}).$$

\end{proof}

By Lemma \ref{lem-differential inequality of distance function}, we get the following diameter estimate for $(\Sigma_r,\bar g)$  (cf. \cite[Proposition 3.3]{DZ6}).

\begin{prop}\label{theorem-diameter estimate}Let $(M,g,f)$ be a steady Ricci soliton as in Lemma \ref{lem-differential inequality of distance function}.   Then there exists a constant $C$ independent of $r$ such that
\begin{align}
{\rm diam}(\Sigma_r,g)\le C\sqrt{r},~\forall~r\ge r_0.
\end{align}

\end{prop}

\begin{rem}
When $R_{\max}=1$, there exists a constant $C_1$ and $C_2$ independent of $r$ such that
\begin{align}
{\rm diam}(\Sigma_r,g)\le C_1\sqrt{r}+C_2,~\forall~r\ge r_0.
\end{align}
Moreover, $C_1$ only depends on $n$ and $C$, where $C$ is the constant in (\ref{uniform curvature decay}).
\end{rem}

As a corollary, we have

\begin{cor}\label{set-mr-contain-2}
Under the condition of Proposition \ref{theorem-diameter estimate},   there exists  a uniform constant $C_0>0$ such that the following is true:
  for any $k\in\mathbb{N}$, there exists $\bar r_0=\bar r_0(k)$ such that
\begin{align}
M_{p,k}\subset B(p,C_0+\frac{2k}{\sqrt{R_{\max}}} ; f^{-1}(p)g), ~\forall ~ \rho(p)\ge  \bar r_0.
\end{align}
\end{cor}

The proof is the same as the one of Corollary 3.4 in \cite{DZ6}. We shall mention  that the definition of $M_{p,k}$ in Corollary \ref{set-mr-contain-2} is a little different from the one  in \cite{DZ6} and it needs to   replace  the scale $R(p)$ there  by $f^{-1}(p)$.  We skip over  the proof.

\section{Proof of Theorem \ref{dimension reduction theorem}}\label{section 4}

The proof of  Theorem \ref{dimension reduction theorem} is based on the  following famous Gromov theorem  \cite{G}.

\begin{theo}\label{Gromov's theorem}
If $V$ is $\widehat{\varepsilon}(n)$-flat, then its universal cover is diffeomorphic to $\mathbb{R}^n$, where $\widehat{\varepsilon}(n)$ is a constant depends only on $n$.
\end{theo}

  The definition of  almost flat manifolds is as follows.

\begin{defi}
Let $(V,g)$ be a connected $n$-dimensional complete Riemannian manifold. Let $d(V)$ be the diameter of $(V,g)$. Let $c^{+}(V)$ and $c^{-}(V)$ be the upper and lower bound of  the sectional curvature of $(V,g)$, respectively. We set $c(V)=\max\{|c^{+}(V)|,|c^{-1}(V)|\}$. We say $V$ is $\varepsilon$-flat for $\varepsilon\ge 0$ if $c(V)d^2(V)\le \varepsilon$.
\end{defi}

With the help of  Theorem \ref{Gromov's theorem},  we will prove

\begin{prop}\label{main theorem}
Let $(M,g,f)$ be an $n(\ge3)$-dimensional steady Ricci soliton with nonnegative sectional curvature and positive Ricci curvature.  Suppose   that   (\ref{scalar curvature linear decay}) holds.
Then  there is a constant $c_0>0$ such that
\begin{align}\label{lower-decay}
R(x)\ge\frac{c_0}{\rho(x)},~\forall ~\rho(x)>>1.
\end{align}
\end{prop}

We first prove two lemmas below.

\begin{lem}\label{lem-decay-1}
Let $(M,g,f)$ be an $n$-dimensional steady Ricci soliton  as in Proposition \ref{main theorem}.
  Suppose that there exists a sequence of points $p_i$  such that  $f(p_i)\to\infty$ and   $R(p_i)f(p_i)\to0$ as $i\to\infty$.  Then there exists a constant $\varepsilon>0$ such that
\begin{align}\label{0-limit}
R(x_i)f(x_i)\to0,~as~i\to\infty,~\forall x_i\in B(p_i,\varepsilon;f^{-1}(p_i)g).
\end{align}
Moreover, $\varepsilon$ only depends on the dimension  $n$  and the constant $C$ in (\ref{scalar curvature linear decay}).
\end{lem}

\begin{proof}
 Let $g_i(t)=f(p_i)^{-1}g(f(p_i)t)$.  We consider the recaled sequence of Ricci flows  $(M,g_i(t),p_i)$.    By Lemma \ref{geodesic ball in level set}, we see that
\begin{align}
B(p_i,1; f(p_i)^{-1}g)\subseteq M_{p_i,\sqrt{R_{\max}}}.
\end{align}
Then, for any $q_i\in B(p_i,1;f(p_i)^{-1}g)$,
\begin{align}\label{limit is 1}
\frac{f(q_i)}{f(p_i)}\to1, ~as~i\to\infty.
\end{align}
Note that   (\ref{linear of f}) holds by Lemma \ref{lem-level set structure nonnegative case}.
 By  the curvature decay (\ref{scalar curvature linear decay}),  it follows that
\begin{align}\label{limit of curvature}
R(q_i)f(p_i)\le 2C, ~as~i\to\infty.
\end{align}
Hence,  the sectional curvature of  $(M,f(p_i)^{-1}g)$ is uniformly bounded by $C(n)C$ on $B(p_i,1;f(p_i)^{-1}g)$ when $i$ large enough.

 Let  $\iota_i:   \widetilde{B}(0,r_0)\subset \mathbb R^n \to  {B}(0,r_0;  f(p_i)^{-1}g)\subset M$  be  an  exponential map such that $\iota_i(0)=p_i$.   Then   there exists a  $r_0\le 1$ (depends only on $C$ and $n$ ) such that  $\iota_i^{\ast}(f(p_i)^{-1}g)(x)$   is a  smooth metric on   $\widetilde{B}(0,r_0)$ which satisfies $\iota_i^{\ast}(f(p_i)^{-1}g)(0)=g_{e}(0)$ and $C^{-1}_0g_{e}\le \iota_i^{\ast}(f(p_i)^{-1}g)(x)\le C_0 g_{e}$ for  any $x\in \widetilde{B}(0,r_0)$.  Here $g_e$ is the Euclidean metric and $C_0$ is a constant independent of $i$. Let $\widetilde{g}_{i}(t)=\iota_i^{\ast}(g_i (t))$.  By Lemma \ref{lem-pointwise curvature estimate} and (\ref{limit is 1}),  it follows that
\begin{align}
|\widetilde{\nabla}^k {\rm \widetilde{Rm}}_{\widetilde{g}_{i}(0)}|_{\widetilde{g}_{i}(0)}(x)=\frac{|\nabla^k {\rm Rm}|(\iota_i(x))}{f(\iota_i(x))^{-\frac{k+2}{2}}}\cdot\frac{f(p_i)^{\frac{k+2}{2}}}{f(\iota_i(x))^{\frac{k+2}{2}}}\le C_1,~\forall~x\in~\widetilde{B}(0,r_0).\notag
\end{align}
Note that
\begin{align}
f(\phi_t(p))-f(p)=\int_{0}^{-t}|\nabla f|^{2}(\phi_s(p))ds\ge 0,~\forall~t\le 0,~p\in M.
\end{align}
 Then,  for any  $t\le0$ and $x\in~\widetilde{B}(0,r_0)$,  we also get
\begin{align}
|\widetilde{\nabla}^k {\rm \widetilde{Rm}}_{\widetilde{g}_{i}(t)}|_{\widetilde{g}_{i}(t)}(x)=&\frac{|\nabla^k {\rm Rm}|(x_i,f(p_i)t)}{f(\phi_{f(p_i)t}(x_i))^{-\frac{k+2}{2}}}\cdot\frac{f(p_i)^{\frac{k+2}{2}}}{f(\phi_{f(p_i)t}(x_i))^{\frac{k+2}{2}}}\notag\\
\le& C_1\cdot\frac{f(p_i)^{\frac{k+2}{2}}}{f(x_i)^{\frac{k+2}{2}}}\notag\\
\le& 2C_1,\notag
\end{align}
where $x_i=\iota_i(x)$.
Hence, by  the Hamilton compactness theorem \cite{H4},  $(\widetilde{B}(0,r_0),\widetilde{g}_{i}(t),0)$ converges   to a limit $(\widetilde{B}(0,r_0),\widetilde{g}_{\infty}(t),0)$
in $\mathcal{C}_{loc}^{\infty}$ sense.

On the other hand,
\begin{align}
\widetilde{R}_{\infty}(0,0)=\lim_{i\to\infty}R(p_i)f(p_i)=0.
\end{align}
 This means  that $\widetilde{R}_{\infty}(x,t)$ attains  its minimum $0$ at the interior point $0\in \widetilde{B}(0,r_0)$,   since $\widetilde{g}_{i}(t)$ has nonnegative sectional curvature.   By the maximum principle,  $(\widetilde{B}(0,r_0),\widetilde{g}(t))$ has constant sectional curvature $0$.  Hence, by  taking $\varepsilon=r_0$,  the convergence  of $(\widetilde{B}(0,r_0),\widetilde{g}_{i}(0),0)$
 will imply  (\ref{0-limit}).

\end{proof}

\begin{lem}\label{lem-decay-2}
Let $(M,g,f)$ be a steady Ricci soliton  and   $\{p_i\}$
  a sequence  such that $f(p_i)\to\infty$ and   $R(p_i)f(p_i)\to0$ as $i\to\infty$  as in Lemma \ref{lem-decay-1}.  Then by taking a subsequence,
\begin{align}
R(x_i)f(x_i)\to0,~as~i\to\infty,~\forall x_i\in \Sigma_{f(p_i)}.
\end{align}
\end{lem}
\begin{proof}
We prove by contradiction. Suppose the lemma is not true. Then, there exists a sequence of point $q_i\to\infty$ and a constant $c'>0$ such that
\begin{align}\label{contradition}
R(q_i)f(q_i)\ge c'.
\end{align}

We consider the sequence of manifolds $(\Sigma_{f(p_i)},f(p_i)^{-1}\bar{g})$, where $\bar{g}$ is the induced metric on $\Sigma_{f(p_i)}$ as a submanifold of $(M,g)$. By  Guass formula  (\ref{gauss-equa})   together with   the curvature estimate  in  Lemma \ref{lem-pointwise curvature estimate},  we  see that $(\Sigma_{f(p_i)},f(p_i)^{-1}\bar{g})$ has uniform  bounded curvature.  Moreover, by Proposition \ref{theorem-diameter estimate}, it has uniform  bounded
diameter estimate.   Thus    $(\Sigma_{f(p_i)},f(p_i)^{-1}\bar{g})$ converges to a length space $(\Sigma_{\infty},\bar{g}_{\infty})$  in  Gromov-Hausdroff topology.  Without loss of  generality,  we may assume that $p_i\to p_{\infty}$ and $q_i\to q_{\infty}$.  Then  $d_{\infty}(p_{\infty},q_{\infty})=L<\infty$.

    Let $N=[\frac{2L}{\varepsilon}]$, where $\epsilon$ is chosen as in Lemma \ref{lem-decay-1}.   Then we can choose $N$ points  $p^1_{\infty},\cdots,p^N_{\infty}$ such that
\begin{align}
d_{\infty}(p_{\infty},p^1_{\infty})&=\varepsilon/2,\notag\\
d_{\infty}(p^k_{\infty},p^{k+1}_{\infty})&=\varepsilon/2,~1\le k\le N-1,\notag\\
d_{\infty}(q_{\infty},p^N_{\infty})&\le\varepsilon/2.\notag
\end{align}
By the convergence of $(\Sigma_{f(p_i)},f(p_i)^{-1}\bar{g})$,  there exists sequences $p_i^{k}$, $1\le k\le N-1$ such that $p_i^k\to p_{\infty}^k$.
Note that  $R(p_i)f(p_i)\to0$ and $p_i^1\in B(p_i,\varepsilon;f^{-1}(p_i)g)$.   Thus by Lemma \ref{lem-decay-1}, we  see that
\begin{align}
R(p_i^1)f(p_i^1)\to 0,~as~i\to\infty.
\end{align}
Since the constant $\varepsilon$ is independent of the choice of sequence of base points ($\{p_i\}$  is replaced by $\{p_i^k\}$ here )  in Lemma \ref{lem-decay-1}, by induction on $k$, we conclude  from  Lemma \ref{lem-decay-1}  that
\begin{align}
R(p_i^k)f(p_i^k)\to 0,~as~i\to\infty, ~1\le k\le N+1.
\end{align}
In particular,   for $k={N+1}$, we get
   $$R(q_i)f(q_i)\to 0,~as~i\to\infty,$$
 which is a contradiction with (\ref{contradition}). The lemma is proved.
\end{proof}

\begin{proof}[Proof of Proposition  \ref{main theorem}]
We prove by contradiction. If the proposition  is not true,   there exists a sequence of point $p_i\to\infty$ such that $R(p_i)f(p_i)\to0$.  Then  by Lemma \ref{lem-decay-2},
\begin{align}
R(x_i)f(x_i)\to0,~as~i\to\infty,~\forall x_i\in \Sigma_{f(p_i)}.\notag
\end{align}
Since the sectional curvature is nonegative,
\begin{align}
|{\rm Rm}(x_i)|f(x_i)\to0,~as~i\to\infty,~\forall x_i\in \Sigma_{f(p_i)}.\notag
\end{align}
 By (\ref{gauss-equa}),  it follows that
\begin{align}
|\overline{\rm Rm}(x_i)|_{f(p_i)^{-1}\bar{g}}\to0, ~as~i\to\infty,~\forall x_i\in \Sigma_{f(p_i)}.\notag
\end{align}
On the other hand, by  Proposition \ref{theorem-diameter estimate},
\begin{align}
{\rm diam}(\Sigma_{f(p_i)},f(p_i)^{-1}\bar{g})\le C.\notag
\end{align}
Hence, we see that $(\Sigma_{f(p_i)},f(p_i)^{-1}\bar{g})$ is a sequence of $\widehat{\varepsilon}(n)$-flat manifold when $i$ is large enough. By Theorem \ref{Gromov's theorem}, the universal cover of $(\Sigma_{f(p_i)},f(p_i)^{-1}\bar{g})$ is diffeomorphic to $\mathbb{R}^{n-1}$. However, $\Sigma_{f(p_i)}$ is diffeomorphic to $\mathbb{S}^{n-1}$  (cf. \cite [Lemma 2.1]{DZ6}).  Therefore,   there is a contradiction.   The  proposition is proved.
\end{proof}

By Proposition \ref{main theorem}, we can finish the proof of Theorem \ref{dimension reduction theorem}.

\begin{proof}[Proof of Theorem \ref{dimension reduction theorem}]
We may suppose  that $(M,g,f)$ is non-flat. Let $(\widetilde{M},\widetilde{g})$ be the universal cover of $(M,g)$ and $\pi:\widetilde{M}\to M$ be the covering map. Let $\widetilde{f}=f\circ \pi$. Then, $(\widetilde{M},\widetilde{g},\widetilde{f})$ is also a steady Ricci soliton. By the proof of Lemma 5.1 in \cite{DZ6}, $(\widetilde{M},\widetilde{g})$ splits as $(N,h)\times\mathbb{R}^{n-k}$ and $\widetilde{f}$ is a constant when restricted on $\mathbb{R}^{n-k}$, where $k\ge 2$. Moreover, $(N,h,f_N)$ is a $k$-dimensional steady Ricci soliton with nonnegative sectional curvature and positive Ricci curvature, where $f_N(q)=\widetilde{f}(q,\cdot)$, $\forall~q\in N$.

 By (4) of Lemma \ref{lem-level set structure nonnegative case} and (\ref{scalar curvature linear decay}), we see that
\begin{align}
R(x)f(x)\le c,~\forall~f(x)\ge r_0.\notag
\end{align}
This implies that
\begin{align}
R_N(x)f_N(x)\le c,~\forall~f_N(x)\ge r_0,\notag
\end{align}
and consequently,
\begin{align}\label{n-scalar}
R_N(x)\rho_N(x)\le c',~\forall~\rho_N(x)\ge r_0'.
\end{align}

When $k=2$,   $(N,h)$ is a two-dimentional  steady Ricci soliton, which is a Cigar soliton \cite{H1, CLN}.
When $k>2$,  by (\ref{n-scalar} ), we can apply Proposition \ref{main theorem} to  the steady Ricci soliton  $(N,h)$  and conclude that  $R_N(x)$ satisfies (\ref{exact linear decay}). The proof is completed.
\end{proof}

Corollary  \ref{cor-classificaion of 3d}  follows from Theorem \ref{dimension reduction theorem}  directly.   We need to consider the case (iii) in the theorem.  However, it is proved  in \cite{DZ5} that
any 3-dimensional steady Ricci soliton   $(N^3, g_N)$ with  scalar curvature decay (\ref{exact linear decay}) must be
rotationally symmetric.

\begin{proof}[Proof of  Corollary \ref{theo-4 dimension noncollapsed case}]   By Theorem \ref{dimension reduction theorem},  the universal cover $(\widetilde{M},\widetilde{g},  \widetilde f)$ of $(M,g, f)$ is  either $(\mathbb{R}^2,g_{cigar})\times(\mathbb{R}^{2},g_{Euclid})$, or  a   4-dimensional steady Ricci soliton $(N^k,g_{N}, f_N)\times(\mathbb{R}^{n-k},g_{Euclid})$ ($k>2$) with the scalar curvature $R_{N}(\cdot)$ of $(N^k,g_{N})$ which satisfying (\ref{exact linear decay}). Since $(M,g)$ is $\kappa$-noncollapsed, $(\widetilde{M},\widetilde{g})$ is also $\kappa$-noncollapsed, and so  it  can't be $(\mathbb{R}^2,g_{cigar})\times(\mathbb{R}^{2},g_{Euclid})$.

In the latter case, $f_N(q)=\widetilde{f}(q,\cdot)$, $\forall~q\in N$. We note that $f_N$ satisfies (\ref{linear of f}) by  Lemma \ref{lem-level set structure nonnegative case} for $N$.  Thus, by (\ref{exact linear decay}),  we have
$$\frac{C^{-1}}{ f_N(x)}\le R_{N}(x)\le\frac{C}{  f_N(x)}, ~ \widetilde f_N(x)\ge r_0.$$
It follows that
$$\frac{C^{-1}}{  \widetilde f(x)}\le R_{\widetilde M}(x)\le\frac{C}{  \widetilde f(x)}, ~ \widetilde f(x)\ge r_0,$$
and consequently,
$$\frac{C^{-1}}{  f(x)}\le R(x)\le\frac{C}{  f(x)}, ~  f(x)\ge r_0.$$
Again by  (\ref{linear of f})  in  Lemma \ref{lem-level set structure nonnegative case}, we get
\begin{align}\label{scalar-curvature-linear}
\frac{(C')^{-1}}{\rho(x)}\le R(x)\le\frac{C'}{  \rho(x)}. ~ \rho(x)\ge r_0'.
\end{align}
This means that  $(M,g)$  is a 4-dimensional    $\kappa$-noncollapsed steady Ricci soliton with nonnegative sectional curvature ,   which satisfies (\ref{scalar-curvature-linear} ).   Hence,  by Theorem 1.5 in \cite{DZ6},   it  must be rotationally symmetric.

\end{proof}

\section{Proof of Theorem \ref{theo-small decay constant}}\label{section 5}

In this section, we generalize the argument in the proof of  Theorem \ref{dimension reduction theorem}
to prove Theorem \ref{theo-small decay constant}. We have assumed that $R_{\max}=1$. Then, by Lemma  \ref{lem-precise estimate of f}  and (\ref{decay with constant e(n)})
\begin{align}
R(x)f(x)\le 2\varepsilon(n),~\forall~f(x)\ge r_0,
\end{align}
when $r_0$ is large enough.

\begin{proof}[Proof of Theorem \ref{theo-small decay constant}]  Let $(\widetilde{M},\widetilde{g},\widetilde{f})$ be the covering steady Ricci soliton of $(M,g,f)$.    Then, by Theorem \ref{dimension reduction theorem},
 $(\widetilde{M},\widetilde{g})$ is a $(\mathbb{R}^2,g_{cigar})\times(\mathbb{R}^{n-2},g_{Euclid})$, or
 $(\widetilde{M},\widetilde{g})=(N, g_{N})\times(\mathbb{R}^{n-k},g_{Euclid})$, $k>2$,   where  $(N,g_{N})$ is a $k$-dimensional steady Ricci soliton with nonnegative sectional curvature  and positive Ricci curvature. Moreover,  there is a positive constant $c_0$ such that the scalar curvature $R_{N}$ of $(N, g_{N})$ satisfies
\begin{align}\label{decay on N}
\frac{c_0}{f_N(x)}\le R_N(x)\le\frac{2\varepsilon(n)}{f_N(x)},~\forall~f_N(x)\ge r_0,
\end{align}
where  $f_N(q)=\widetilde{f}(q,\cdot)$, $\forall~q\in N$. Hence, to prove  Theorem \ref{theo-small decay constant}, it suffices to eliminate the second  case  $(N^k,g_{N})\times(\mathbb{R}^{n-k},g_{Euclid})$ by contradiction.

Let $\overline{\Sigma}_r=\{x\in N|f_N(x)=r\}$ be the level set of $f_N$. By Proposition \ref{theorem-diameter estimate}, the diameter $D_r$ of $(\overline{\Sigma}_r,r^{-1}h)$ satisfies
\begin{align}
D_r\le C_1(n),~\forall~r\ge r_1, \notag
\end{align}
where $r_1$ is a large number and $C_1(n)$ is a uniform constant depending only on the dimension $n$.
Thus,  as in the proofs of  Lemma \ref{geodesic ball in level set} and Corollary \ref{set-mr-contain-2},  we see that
\begin{align}
\overline{\Sigma}_r\subseteq B(x,2C_1(n);r^{-1}h)\subseteq M_x,2C_1(n),~\forall~x\in \overline{\Sigma}_r.\notag
\end{align}

On the other hand, by  (\ref{decay on N}) and the nonnegativity of sectional curvature,
\begin{align}
|{\rm Rm}_{h}|_h(x)\cdot r\le C_2(n)\varepsilon(n),~\forall~x\in \overline{\Sigma}_r.\notag
\end{align}
Then by Lemma \ref{lem-pointwise curvature estimate} and Lemma \ref{lem-formula of sectional curvature} together with
 (\ref{gauss-equa}), we get
\begin{align}
|\overline{\rm Rm}_{r^{-1}h}|_{r^{-1}h}=r\cdot|\overline{\rm Rm}_h|_h\le  2r\cdot|{\rm Rm}_h|_h(x)\le 2C_2(n)\varepsilon(n),~\forall~x\in \overline{\Sigma}_r.\notag
\end{align}
It follows that
\begin{align}
|\overline{\rm Rm}_{r^{-1}h}|_{r^{-1}h}\cdot D^2_r\le 2C_1(n)C^2_2(n)\varepsilon(n),~\forall~x\in \overline{\Sigma}_r,~r\ge r_1.\notag
\end{align}
When $\varepsilon(n)\le  \frac{1}{2}C^{-1}_1(n)C^{-2}_2(n)\widehat{\varepsilon}(n)$, $(\overline{\Sigma}_r,r^{-1}h)$ is an $\widehat{\varepsilon}(n)$-flat metric.  Hence,   by Gromov's theorem, Theorem \ref{Gromov's theorem},  the universal cover of $\overline{\Sigma}_r$ must be diffeomorphic to $\mathbb{R}^{k-1}$.  However,   $\overline{\Sigma}_r$ is diffeomorphic to $\mathbb{S}^{k-1}$ (cf. \cite [Lemma 2.1]{DZ6}) .   Therefore,  we get a contradiction.   As a consequence,  $(\widetilde{M},\widetilde{g})$ must be  $(\mathbb{R}^2,g_{cigar})\times(\mathbb{R}^{n-2},g_{Euclid})$.

\end{proof}

\section{ A generalization of Proposition  \ref{main theorem} }\label{section 6}

In this section,   we prove an analogy of  Proposition  \ref{main theorem} for $\kappa$-noncollapsed steady Ricci solitons  which satisfying(\ref{uniform curvature decay}).  Namely, we show

\begin{theo}\label{theorem-uniform decay and noncollapsed case}
Let $(M^n,g,f)$ be a $\kappa$-noncollapsed steady Ricci soliton which  satisfies (\ref{uniform curvature decay}).  Suppose that $f$ satisfies (\ref{linear of f}). Then $\Sigma_r$ is diffeomorphic to a compact shrinking Ricci soliton with nonnegative Ricci curvature for any $r \ge r_1>0$. Moreover, there is $c_0>0$ such that
\begin{align}\label{lowe estimate of R}
R(x)\ge\frac{c_0}{f(x)},~\forall~\rho(x)>>1.
\end{align}
\end{theo}

Because of lack of positivity of sectional curvature, we could not use Gromov's Theorem  to study the structure
of level sets on  steady Ricci soliton as in Section 4. Here we  first use the $\kappa$-noncollapsed condition to derive the convergence of  rescaled flows $(M,f^{-1}(p_i)g(f(p_i)t),p_{i})$ with the help of diameter estimate established in Section 3.

\begin{lem}\label{theo-convergence of genneral case}
Let $(M,g,f)$ be a  $\kappa$-noncollapsed steady Ricci soliton  as in  Theorem \ref{theorem-uniform decay and noncollapsed case}.  Then, for any $p_{i}\rightarrow \infty$,
 rescaled flows $(M,f^{-1}(p_i)g(f(p_i)t),p_{i})$ converges subsequently to
$(\mathbb{R}\times \Sigma,$ $ds^2+g_{\Sigma}(t))$ ( $t\in (-\infty,0]$) in the Cheeger-Gromov topology, where   $\Sigma$ is diffeormorphic to  the level set $\Sigma_{r_0}$ of $f$ and
 $(\Sigma,$ $g_{\Sigma}(t))$ is a $\kappa$-noncollapsed ancient Ricci flow with nonnegative Ricci curvature. Moreover, the scalar curvature $R_{\Sigma}(t)$ satisfies
 \begin{align}\label{monotonicity of R}
 \frac{\partial}{\partial t}R_{\Sigma}(x,t)\ge0, ~\forall~x\in \Sigma,~t\le 0,
 \end{align}
 and the curvature tensor ${\rm Rm}_{\Sigma}(x,t)$ of $(\Sigma,g_{\Sigma}(t))$ satisfies
 \begin{align}\label{decay of t}
 |{\rm Rm}_{\Sigma}|(x,t)\le\frac{C}{1+|t|}, ~\forall~x\in \Sigma,~t\le0,
 \end{align}
 where $C$ is a uniform constant.
\end{lem}

\begin{proof}  The proof  is a modification of \cite[Theorem 1.4]{DZ6}, while at present there is no assumption of nonnegativity of  sectional curvature.  For any fixed $\bar{r}>0$,  by  (\ref{set-mr-contain-1}), we have
\begin{align}
 B(p_i,\bar{r};f^{-1}(p_i)g(0))\subseteq M_{p_i,\bar{r}\sqrt{R_{\max}}}.\notag
\end{align}
Moreover, by the relation,
$$\frac{{\rm d}|\nabla f|^2(\phi_{t}(p))}{{\rm d}t}=-2{\rm Ric}(\nabla f,\nabla f)(\phi_{t}(p))\le 0,~\forall~t\le0, ~f(p)\ge r_0,$$
 one can show that for any $t\le0$,
\begin{align}\label{lower bound of ||}
|\nabla f|^2(\phi_{t}(x))\ge |\nabla f|^2(x)\ge \frac{R_{\max}}{2},~\forall ~x\in B(p_i,\bar{r};f^{-1}(p_i)g(0)).
\end{align}

For $x\in B(p_i,\bar{r};f^{-1}(p_i)g(0))$, by  (\ref{uniform curvature decay}), it is easy to see that
\begin{align}
|{\rm Rm}_{g_{p_i}(t)}(x)|_{g_{p_i}(t)}=|{\rm Rm}(x,f(p_i)t)|\cdot f(p_i)\le \frac{Cf(p_i)}{f(\phi_{f(p_i)t}(x))}.\notag
\end{align}
Moreover,  by (\ref{lower bound of ||}),
\begin{align}
f(x,f(p_i)t)-f(x)=\int_{0}^{f(p_i)|t|}|\nabla f|^2(\phi_s(x))ds\ge \frac{R_{\max}}{2}\cdot f(p_i)|t|.\notag
\end{align}
Thus
\begin{align}
|{\rm Rm}_{g_{p_i}(t)}(x)|_{g_{p_i}(t)}\le \frac{Cf(p_i)}{f(x)+\frac{R_{\max}}{2}\cdot f(p_i)|t|}\le \frac{2C}{1+R_{\max}|t|},~\forall~t\le0.\notag
\end{align}
As in the proof of Lemma \ref{lem-pointwise curvature estimate}, for any $t\le0$, we  further get
\begin{align}
|\nabla^k{\rm Rm}_{g_{p_i}(t)}(x)|_{g_{p_i}(t)}\le \frac{2C(k)}{1+R_{\max}|t|},~\forall~x\in B(p_i,\bar{r};f^{-1}(p_i)g(0)),\notag
\end{align}
where $C(k)$ is independent of $i$ and $\bar{r}$.
Note that  $(M,g_{p_i}(t))$ is $\kappa$-noncollapsed.  Hence,  $(M,g_{p_i}(t),p_i)$ converge subsequently to a Ricci flow $(M_{\infty},g_{\infty}(t),p_{\infty})$ with nonnegative Ricci curvature in   Cheeger-Gromov topology, which  satisfies the curvature decay
\begin{align}\label{t-decay}
 |{\rm Rm}_{\infty}(x,t)|\le\frac{C}{1+|t|}, ~\forall~x\in M_{\infty},~t\le 0.
 \end{align}
 On the other hand, for any  $~t\le0$ and $x\in B(p_i,\bar{r};f^{-1}(p_i)g(0))$, it holds
\begin{align}
&\frac{\partial R_{g_{p_i}(t)(x)}}{\partial t}=f(p_i)^2\frac{\partial R}{\partial t}(x,f(p_i)t)\notag\\
&=2f(p_i)^2{\rm Ric}(\nabla f,\nabla f)(x,f(p_i)t)\ge0.\notag
\end{align}
Therefore,  we obtain  (\ref{monotonicity of R}) from (\ref{t-decay}) immediately.

 Let $X_{(i)}=f(p_{i})^{\frac{1}{2}}\nabla f$.
\begin{align}
\sup_{ B(p_{i},\bar r ;  {g_{p_i}})}| \nabla_{(g_{p_i})}X_{(i)}|_{g_{p_i}}&= \sup_{ B(p_{i},\bar r ;  {g_{p_i}})}|{\rm Ric} |\cdot\sqrt{f(p_{i})}
\le C\sqrt{R(p_{i})} \to 0.\notag
\end{align}
Then,  by Lemma \ref{lem-pointwise curvature estimate},
$$\sup_{ B(p_{i},\bar r;  {g_{p_i}})}| \nabla^{k}_{(g_{p_i})}X_{(i)}|_{g_{p_i}}\leq C(n)\sup_{ B(p_{i},\bar r;  {g_{p_i}})}| \nabla^{k-1}_{(g_{p_i})}{\rm Ric}({g_{p_i})}|_{g_{p_i}}\le C_1.$$
Hence,  $X_{(i)}$ converges  subsequently  to a parallel  vector field $X_{(\infty)}$ on $( M_{\infty},$ $  g_{\infty}(0))$.
 Moreover,
 \begin{align}
 |X_{(i)}|_{g_{p_i}}( x)=|\nabla f|(p_{i})
=\sqrt{R_{\rm max}}+o(1)>0, ~\forall~ x\in B(p_{i},\bar r ;  {g_{i}}),\notag
 \end{align}
as long as $f(p_i)$ is large enough.  This implies that $X_{(\infty)}$ is non-trivial. As a consequence,
   $( M_{\infty},g_{\infty}(t))$ locally splits off a piece of  line along $X_{(\infty)}$. By the same argument in the proof of Lemma 4.6 in \cite{DZ6}, we can further  show that  $X_{(\infty)}$
 generates a line through $p_\infty$.  Hence, $(M_{\infty},g_{\infty}(0))$ splits off a line.

 Now, we may assume that $(M_{\infty},g_{\infty}(t))=(\mathbb{R}\times \Sigma,$ $ds^2+g_{\Sigma}(t))$ for $t\le0$. Consequently,  the curvature ${\rm Rm}_{\Sigma}(x,t)$ of $(\Sigma,g_{\Sigma}(t))$ satisfies (\ref{decay of t}). We are left to show that $\Sigma$ is diffeomorphic to $\Sigma_{r_0}$ when $r\ge r_0$.
 It suffices to prove  the convergence of  $(\Sigma_{f(p_i)},R(p_i)\bar{g})$  as in \cite{DZ6}.   In fact,  all  Lemma 4.2-4.4, Proposition 4.5 and Lemma 4.7 in \cite[Section 4]{DZ6}  are valid,   while the corresponding estimates  have been obtained    here in Lemma \ref{geodesic ball in level set}, Lemma \ref{lem-pointwise curvature estimate}, Proposition \ref{theorem-diameter estimate} and Corollary \ref{set-mr-contain-2}.  Thus  $(\Sigma_{f(p_i)},f(p_i)^{-1}\bar{g})$ converge subsequently  to a compact manifold  $(\Sigma_{\infty},\bar{g}_{\infty})$ in Cheeger-Gromov sense. Moreover,  $(\Sigma_{\infty},\bar{g}_{\infty})\subseteq \Sigma\times\{p_{\infty}\}$. By the connectness of $\Sigma$,   $\Sigma_{\infty}$ is diffeomorphic to  $\Sigma$, and so is each $\Sigma_{r}$.    The lemma is  proved.

\end{proof}

The following result seems well-known.

\begin{lem}\label{compact shrinking soliton not flat}
 The scalar curvature of a compact gradient shrinking Ricci soliton does not vanish.
\end{lem}

\begin{proof}
Let $(\Sigma,g,f)$ be a compact gradient shrinking Ricci soliton, i.e., it satisfies
\begin{align}
{\rm Ric}-\frac{g}{2}+{\rm Hess}f=0.\notag
\end{align}
By the flat condition, it follows that
\begin{align}\label{laplacian of f}
\Delta f=\frac{n}{2}.
\end{align}
Since $\Sigma$ is compact, $f$ attains its maximum at some point $x_0\in \Sigma$. Hence,
\begin{align}
\Delta f(x_0)\le0.\notag
\end{align}
It is impossible by (\ref{laplacian of f}). We complete the proof.
\end{proof}

\begin{proof}[Proof of Theorem \ref{theorem-uniform decay and noncollapsed case}]
We first show that $\Sigma$ in Lemma \ref{theo-convergence of genneral case} is diffeomorphic to a compact shrinking Ricci soliton with
nonnegative Ricci curvature when $r\ge r_0$. Since
 $(\Sigma,$ $g_{\Sigma}(t))$ is a $\kappa$-noncollapsed ancient Ricci flow with nonnegative Ricci curvature satisfying (\ref{decay of t}),
by   Theorem 3.1 in \cite{Na},  we see that for any fixed $x\in \Sigma$ and  any sequence $\{\tau_i\}\to\infty$,
 $(\Sigma,\tau_i^{-1} g_{\Sigma}(\tau_i t),x)$ subsequently converges to a shrinking Ricci soliton $(\Sigma',$ $g'(t),x')$ with nonnegative Ricci curvature.  On the other hand, by Lemma 0.3 in \cite{Ni},  there exists a constant $C_1$  such that
 \begin{align}
 {\rm diam}( \Sigma,g_{\Sigma}(t))\le C_1\sqrt{|t|+1}.\notag
 \end{align}
 In particular,
 \begin{align}\label{diam estimate}
 {\rm diam}(  \Sigma, \tau_i^{-1} g_{\Sigma}(-\tau_i))\le C_1.
 \end{align}
    Thus,  $\Sigma'$ is diffeomorphic to $\Sigma$ by Cheeger-Gromov compactness theorem.  Consequently, each  $\Sigma_{r}$  is diffeomorphic to $\Sigma'$ for any $ r\ge r_1$.

    Next, we prove (\ref{lowe estimate of R}). Suppose that (\ref{lowe estimate of R}) is not true. Then, there exists $p_i\to\infty$ such that
    \begin{align}\label{limit is zero}
    R(p_i)f(p_i)\to0,~as~i\to\infty.
    \end{align}
 By Lemma \ref{theo-convergence of genneral case}, for any $p_{i}\rightarrow \infty$,
 rescaled flows $(M,f^{-1}(p_i)g(f(p_i)t),p_{i})$ converges subsequently to
$(\mathbb{R}\times \Sigma,$ $ds^2+g_{\Sigma}(t),p_{\infty})$ ( $t\in (-\infty,0]$) in the Cheeger-Gromov topology. By (\ref{limit is zero}), $R_{\Sigma}(p_{\infty},0)=0$. Since $g_{\Sigma}(t)$ is ancient, $R_{\Sigma}(t)$ is nonnegative by \cite{Ch}. Combining with (\ref{monotonicity of R}), $R_{\Sigma}(p_{\infty},t)=0$, $\forall~t\le0$. By the maximum principle, we see that $R_{\Sigma}(x,t)=0$, $\forall~t\le0$, $x\in\Sigma$. Note that $(\Sigma,\tau_i^{-1} g_{\Sigma}(-\tau_i),p_{\infty})$ subsequently converges to a shrinking Ricci soliton $(\Sigma',$ $g'(-1),p')$ for some $\{\tau_i\}\to\infty$. Hence, we get a compact shrinking Ricci soliton $(\Sigma',$ $g'(-1))$ whose scalar curvature is zero. This is impossible by Lemma \ref{compact shrinking soliton not flat}. Hence, we complete the proof.
\end{proof}

\subsection{Proof of  Theorem \ref{cor-4d-Ricci positive}}
According to \cite{Br2},

\begin{defi}\label{brendle} An $n$-dimensional steady  Ricci soliton $(M,g,f)$  is called  asymptotically   cylindrical   if   the following
holds:

(i) Scalar curvature $R(x)$  of $g$ satisfies
$$\frac{C_1}{\rho(x)} \le  R(x) \le  \frac{C_2}{\rho(x)},~\forall~\rho(x)\ge r_0, $$
where $C_1, C_2$  are two positive constants.

(ii) Let $p_i$ be an arbitrary sequence of marked points going to infinity.
Consider  rescaled metrics
$ g_i(t) = r_i^{-1} \phi^*_{r_i t} g,$ where
$r_i R(p_i) = \frac{n-1}{2} + o(1)$ and $ \phi_{ t}$ is a one-parameter subgroup generated by $X=-\nabla f$.   As  $i \to\infty,$
 flows $(M,  g_i(t), p_i)$
converge in the Cheeger-Gromov sense to a family of shrinking cylinders
$(   \mathbb R \times \mathbb S^{n-1}(1), \widetilde g(t)), t \in  [0, 1).$  The metric $\widetilde g(t)$  is given by
\begin{align} \widetilde g(t) =   dr^2+ (n - 2)(2 -2t) g_{\mathbb S^{n-1}(1)},\notag
\end{align}
 where $\mathbb S^{n-1}(1)$ is the unit sphere in Euclidean space.
\end{defi}

 We need to describe some properties of  asymptotically cylindrical geometry on steady Ricci solitons.

\begin{lem}\label{lem-limit of sectional curvature}
Let $(M,g, f)$ be a noncompact steady Ricci soliton which is asymptotically cylindrical. Then, we have
\begin{align}
\frac{|{\rm Hess}R|+|\Delta{\rm Ric}|}{R^2}\to0,~as~\rho(x)\to\infty, \notag\\
\frac{{\rm Ric}(e_i,e_j)}{R}\to\frac{\delta_{ij}}{n-1},~as~\rho(x)\to\infty,\notag\\
\frac{{\rm Ric}(e_i,e_n)}{R}\to0,~as~\rho(x)\to\infty,\notag\\
\frac{{\rm Rm}(e_j,e_k,e_k,e_l)}{R}\to\frac{(1-\delta_{jk})\delta_{jl}}{(n-1)(n-2)},~as~\rho(x)\to\infty,\notag\\
\frac{{\rm Rm}(e_n,e_i,e_j,e_k)}{R}\to0,~as~\rho(x)\to\infty,\notag
\end{align}
where $\{e_1,\cdots,e_{n}\}$ are the orthonormal basis of $T_xM$ with respect to metric $g$ and $1\le i,j,k,l\le n-1$, $e_n=
\frac{\nabla f}{|\nabla f|}$.
\end{lem}
\begin{proof}
We only need to show that for any $p_i$ tends to infinity and for any sequence of orthonormal basis $\{e_1^{(i)},\cdots,e_{n-1}^{(i)}\}\subseteq T_{p_i}\Sigma_{f(p_i)}$,  the following holds (perhaps after taking a subsequence),

\begin{align}
\frac{|{\rm Hess}R|(p_i)+|\Delta{\rm Ric}|(p_i)}{R^2(p_i)}\to0,~as~i\to\infty,\\
\frac{{\rm Ric}(e^{(i)}_j,e^{(i)}_k)}{R(p_i)}\to\frac{\delta_{jk}}{n-1},~as~i\to\infty,\label{ric-limit}\\
\frac{{\rm Ric}(e^{(i)}_j,e^{(i)}_n)}{R(p_i)}\to0,~as~i\to\infty,\label{ric-limit-n}\\
\frac{{\rm Rm}(e^{(i)}_j,e^{(i)}_k,e^{(i)}_k,e^{(i)}_l)}{R(p_i)}\to\frac{(1-\delta_{jk})\delta_{jl}}{(n-1)(n-2)},~as~i\to\infty, \label{Rm-limit}\\
\frac{{\rm Rm}(e^{(i)}_n,e^{(i)}_j,e^{(i)}_k,e^{(i)}_l)}{R(p_i)}\to0,~as~i\to\infty. \label{Rm-limit-n}
\end{align}

For any $p_i$ tends to infinity, we consider the sequence of pointed manifold $(M,g_{p_i},p_i)$, where $g_{p_i}=R(p_i)g$ (the scale $r_i^{-1}$ is replaced by $\frac{2}{n-1}r_i^{-1}$  in Definition \ref{brendle}). By the asymptotically cylindrical condition, by taking a subsequence, we have $(M,g_{p_i},p_i)$ converge to $(\mathbb{R}\times\mathbb{S}^{n-1},g_{\infty},p_{\infty})$, where $g_{\infty}=ds^2+(n-1)(n-2)g_{\mathbb{S}^{n-1}(1)}$. Then it is easy to see that
\begin{align}
&\lim_{i\to\infty}\frac{|{\rm Hess}R|(p_i)+|\Delta{\rm Ric}|(p_i)}{R^2(p_i)}\notag\\
=&\lim_{i\to\infty}\big(|{\rm Hess}R_{g_{p_i}}|_{g_{p_i}}(p_i)+|\Delta_{g_{p_i}}{\rm Ric}_{g_{p_i}}|(p_i)\big)\notag\\
=&|{\rm Hess}R_{g_{\infty}}|_{g_{\infty}}(p_{\infty})+|\Delta_{g_{\infty}}{\rm Ric}_{g_{\infty}}|(p_{\infty})\notag\\
=&0.\notag
\end{align}
Moreover, for any fixed $\overline{r}>0$, we have
\begin{align}\label{Ric-convergence}
\lim_{i\to\infty}\frac{|{\rm Ric}|(x)}{R(p_i)}=\lim_{i\to\infty}|{\rm Ric}_{g_{p_i}}|_{g_{p_i}}(x)=\frac{1}{\sqrt{n-1}},~\forall~x\in B(p_{i},\bar r;  {g_{p_i}}).
\end{align}

Let $X_{(i)}=R(p_{i})^{-\frac{1}{2}}\nabla f$. By (\ref{Ric-convergence}), we see that
\begin{align}
\sup_{ B(p_{i},\bar r ;  {g_{p_i}})}| \nabla_{(g_{p_i})}X_{(i)}|_{g_{p_i}}&= \sup_{ B(p_{i},\bar r ;  {g_{p_i}})}\frac{|{\rm Ric}|}{\sqrt{R(p_{i})}}\le 2\sqrt{\frac{R(p_{i})}{n-1}} \to 0.\notag
\end{align}
Note that
$$\sup_{ B(p_{i},\bar r;  {g_{p_i}})}| \nabla^{m}_{(g_{p_i})}X_{(i)}|_{g_{p_i}}\leq C(n)\sup_{ B(p_{i},\bar r;  {g_{p_i}})}| \nabla^{m-1}_{(g_{p_i})}{\rm Ric}({g_{p_i})}|_{g_{p_i}}\to 0.$$
Then $X_{(i)}$ converges to a parallel  vector field $X_{(\infty)}$ on $( M_{\infty},$ $  g_{\infty})$.
 Moreover,
 \begin{align}
 |X_{(i)}|_{g_{p_i}}( x)=|\nabla f|(p_{i})
=\sqrt{R_{\rm max}}+o(1)>0, ~\forall~ x\in B(p_{i},\bar r ;  {g_{p_i}}).\notag
 \end{align}
   This implies that $X_{(\infty)}$ is non-trivial. Hence, $X_{(\infty)}$ is tangent to $\mathbb{R}$ in $\mathbb{R}\times\mathbb{S}^{n-1}$.

   Suppose that
   $R(p_i)^{-\frac{1}{2}}e^{(i)}_j\to e^{(\infty)}_j$ for $1\le j\le n-1$. Then,
   \begin{align}
   e^{(\infty)}_n=\lim_{i\to\infty}\frac{X_{(i)}}{|X_{(i)}|_{g_{p_i}}}=\frac{X_{(\infty)}}{|X_{(\infty)}|_{g_{\infty}}}, \notag\\
   \langle e^{(\infty)}_j,X_{(\infty)}\rangle_{g_{\infty}}=\lim_{i\to\infty}\langle e^{(i)}_j, X_{(i)} \rangle_{g_{p_i}}=0,\notag\\
   |e^{(\infty)}_j|_{g_{\infty}}=\lim_{i\to\infty}R(p_i)^{-\frac{1}{2}}|e^{(i)}_j|_{g_{p_i}}=1.\notag
   \end{align}
  Thus,  all unit vectors $e^{(\infty)}_j\in T\mathbb{S}^{n-1}$ for $1\le j\le n-1$. It follows that
   \begin{align*}
   {\rm Ric}_{g_{\infty}}(e^{(\infty)}_j,e^{(\infty)}_k)=\frac{\delta_{jk}}{n-1},~\forall~1\le j,k\le n-1.\\
   {\rm Ric}_{g_{\infty}}(e^{(\infty)}_j,e^{(\infty)}_n)=0,~\forall~1\le j\le n-1.\\
   {\rm Rm}_{g_{\infty}}(e^{(\infty)}_j,e^{(\infty)}_k,e^{(\infty)}_k,e^{(\infty)}_l)=\frac{(1-\delta_{jk})\delta_{jl}}{(n-1)(n-2)},~\forall~1\le j,k,l\le n-1,\\
   {\rm Rm}_{g_{\infty}}(e^{(\infty)}_n,e^{(\infty)}_j,e^{(\infty)}_k,e^{(\infty)}_l)=0,~\forall~1\le j,k,l\le n-1.
   \end{align*}
Hence,  by the convergence of $(M,g_{p_i},p_i)$, it is easy to check that (\ref{ric-limit}), (\ref{ric-limit-n}),  (\ref{Rm-limit}) and (\ref{Rm-limit-n}) are all true.

\end{proof}

\begin{lem}\label{lem-positive sectional curvature}
Let $(M,g,f)$ be a noncompact steady Ricci soliton with positive Ricci curvature, which is asymptotically cylindrical. Then, there exists a compact set $K\subseteq M$ such that  $(M,g)$ has positive sectional curvature on $M\setminus K$.
\end{lem}

\begin{proof}
For $p\in M$, let $\{e_1,\cdots,e_n\}$ be an  orthonormal basis of $T_{p}M$ with respect to metric $g$ and $e_n=\frac{\nabla f}{|\nabla f|}$. When $\rho(p)\ge r_0$,  we have
\begin{align}
{\rm Rm}(e_j,e_k,e_k,e_j)\ge \frac{R(p)}{2(n-1)(n-2)}>0,~\forall~1\le j<k\le n-1.\notag
\end{align}
By Lemma \ref{lem-formula of sectional curvature} and Lemma \ref{lem-limit of sectional curvature}, for any $\epsilon>0$, we have
\begin{align}
&|\nabla f|^2(p)\cdot{\rm Rm}(e_n,e_j,e_j,e_n)\notag\\
&=-\frac{1}{2}({\rm Hess}R)_{jj}-\sum_{l=1}^n R_{jl}R_{jl}+\Delta R_{jj}+2\sum_{i,l=1}^{n}R_{ijjl}R_{il}\notag\\
&\ge (\frac{1}{(n-1)^2}-\epsilon)R^2(p)>0.\notag
\end{align}
Hence, $(M,g)$ has positive sectional curvature on $M\setminus B(p_0,r_0;g)$ for some $p_0\in M$ and $r_0>0$.
\end{proof}

 \begin{proof}[Proof of  Theorem \ref{cor-4d-Ricci positive}]
 Since $R(x)$ decays uniformly, $R(x)$ attains its maximum at some point $o$. Thus
 \begin{align}
 {\rm Ric}( \nabla f,\nabla f)(o)=-\frac{1}{2}\langle\nabla R,\nabla f\rangle(o)=0.\notag
 \end{align}
 Note that ${\rm Ric}(\cdot)$ is positive, we get $\nabla f(o)=0$. By \cite{CaCh}, we know that $f$ satisfies (\ref{linear of f}).

 By  Lemma  \ref{theo-convergence of genneral case}, for any $p_{i}\rightarrow \infty$,
 rescaled flows $(M,f^{-1}(p_i)g(f(p_i)t),p_{i})$ converges subsequently to
$(\mathbb{R}\times \Sigma,$ $ds^2+g_{\Sigma}(t))$ ( $t\in (-\infty,0]$) in the Cheeger-Gromov topology. On the other hand,
    by  (\ref{lowe estimate of R}) in Theorem \ref{theorem-uniform decay and noncollapsed case}, we have
 \begin{align}\label{equivalence of f and R}
 \frac{C}{f(x)}\ge|{\rm Rm}|(x)\ge R(x)\ge \frac{c}{f(x)},~\forall~\rho(x)\ge r_0.
 \end{align}
  This means that  $f$ and
 $R^{-1}$ is  equivalent.  Thus, we may assume that $f(p_i)R(p_i)\to C_0$, as $i\to\infty$ by taking subsequence. As a consequence,  the rescaled flows $(M,R(p_i)g(R^{-1}(p_i)t),p_{i})$ converges subsequently to
$(\mathbb{R}\times \Sigma,$ $ds^2+g'_{\Sigma}(t))$ ( $t\in (-\infty,0]$) in the Cheeger-Gromov topology, where $g'_{\Sigma}(t)=C_0g_{\Sigma}(C_0^{-1}t)$. Note that $(\Sigma,g'_{\Sigma}(t))$ is a 3-dimensional $\kappa$-noncollpased Ricci flow which satisfies $(\ref{decay of t})$. By \cite{Ch}, $(\Sigma,g'_{\Sigma}(t))$ has nonnegative sectional curvature.  Since the Ricci curvature is positive, the level set of $f$ is diffeomorphic to $\mathbb{S}^3$ by \cite{DZ6}. Hence, $\Sigma$ is diffeomorphic to $\mathbb{S}^3$. By \cite{Ni}, $(\Sigma,g'_{\Sigma}(t))$ is a group of shrinking sphere. Therefore, we prove that $(M,g,f)$ is asymptotically clinderical. Moreover, $(M,g)$ has positive sectional curvature on $M\setminus K$ for some compact $K$  by Lemma \ref{lem-positive sectional curvature}.
 
  \end{proof}


\vskip6mm




\section*{References}

\small

\begin{enumerate}


\renewcommand{\labelenumi}{[\arabic{enumi}]}

\bibitem{Br1} Brendle, S., \textit{Rotational symmetry of self-similar solutions to the Ricci flow}, Invent. Math. , \textbf{194} No.3 (2013), 731-764.

\bibitem{Br2} Brendle, S., \textit{Rotational symmetry of Ricci solitons in higher dimensions}, J. Diff. Geom., \textbf{97} (2014), no. 2, 191-214.




\bibitem{Ca} Cao, H.D., \textit{Existence of gradient K\"{a}hler-Ricci solitons}, Elliptic and parabolic methods in geometry (Minneapolis, MN, 1994), 1-16, A K Peters, Wellesley, MA, 1996.


\bibitem{CaCh} Cao, H.D. and Chen, Q., \textit{On locally conformally flat gradient steady Ricci solitons},
Trans. Amer. Math. Soc., \textbf{364} (2012), 2377-2391 .


\bibitem{C-He} Cao, H.D. and He, C.X., \textit{Infinitesimal rigidity of collapsed gradient steady Ricci solitons in dimension three}, arXiv:math/1412.2714v1.

\bibitem{Ch} Chen, B.L., \textit{Strong uniqueness of the Ricci flow},  J. Diff.  Geom. \textbf{82} (2009),  363-382.



\bibitem{CLN} Chow, B., Lu, P. and Ni, L., \textit{Hamilton's Ricci flow} in: {\it Lectures in Contemporary Mathematics 3}
,Science Press, Beijing $\&$ American Mathematical Society, Providence, Rhode Island (2006).






\bibitem{DW}  Dancer, A. and Wang, M., \textit{Some new examples of non-Ka¡§hler Ricci solitons}, Math. Res. Lett. \textbf{16}, (2009) 349-363.


\bibitem{DZ1}Deng, Y.X. and Zhu, X.H., \textit{Complete non-compact gradient Ricci solitons with nonnegative Ricci curvature},
Math. Z., \textbf{279} (2015), no. 1-2, 211-226.

\bibitem{DZ2} Deng, Y.X. and Zhu, X.H., \textit{Asymptotic behavior of positively curved steady Ricci solitons}, Trans. Amer. Math. Soc. \textbf{370} (2018), no.4, 2855-2877.

\bibitem{DZ5} Deng, Y.X.; Zhu, X.H., 3D steady gradient Ricci solitons with linear curvature decay, arXiv:math/1612.05713, to appear in IMRN.

\bibitem{DZ6} Deng, Y.X. and Zhu, X.H., \textit{Higher dimensional steady Ricci solitons with linear curvature decay}, arXiv:math/1710.07815.

\bibitem{G} Gromov, M., \textit{Almost flat manifolds}, J. Diff. Geom., \textbf{13} (1978), 231-241.

\bibitem{GLX} Guan, P.F., Lu, P. and Xu, Y.Y., \textit{A regidity theorem for codimension one shrinking gradient Ricci solitons in $\mathbb{R}^{n+1}$},
 Calc. Var. Partial Differential Equations \textbf{54} (2015), no. 4, 4019-4036.

\bibitem{H2} Hamilton, R.S., \textit{Three manifolds with positive Ricci curvature}, J. Diff. Geom., \textbf{17} (1982), 255-306.


\bibitem{H4} Hamilton, R.S., \textit{A compactness property for solution of the Ricci flow}, Amer. J. Math., \textbf{117} (1995), 545-572.

\bibitem{H1} Hamilton, R.S., \textit{Formation of singularities in the Ricci flow}, Surveys in Diff. Geom., \textbf{2} (1995),
7-136.



\bibitem{MSW} Munteanu, O., Sung, C.-J. A. and  Wang, J., \textit{Poisson Equation on complete manifolds}, arXiv:math/1701.02865v1.


\bibitem{Ni} Ni, L., \textit{Closed type-I Ancient solutions to Ricci flow}, Recent Advances in Geometric Analysis, ALM, vol. 11 (2009), 147-150.

\bibitem{Na} Naber, A., \textit{Noncompact shrinking four solitons with nonnegative curvature}, J. Reine Angew Math., \textbf{645} (2010), 125-153.


\bibitem{Pe1} Perelman, G., \textit{The entropy formula for the Ricci flow and its geometric applications}, arXiv:math/0211159.




\bibitem{S1} Shi, W.X., \textit{Ricci deformation of the metric on complete noncompact Riemannian
manifolds}, J. Diff. Geom., \textbf{30} (1989), 223-301.


\end{enumerate}
\end{document}